\documentclass[reqno, 15pt]{amsart}
\usepackage{latexsym}
\usepackage{amsmath}
\usepackage{amssymb}
\usepackage{amsthm}
\usepackage{amscd}
\usepackage{graphicx}
\usepackage{xcolor}
\usepackage[colorlinks,citecolor=blue]{hyperref}
\usepackage{tikz}
\usepackage{tikz-cd}
\usepackage{caption} 
\usepackage[a4paper,top=3cm,bottom=3cm,left=3cm,right=3cm]{geometry}
\usepackage{cite}
\usepackage{enumitem}
\usepackage{tabularx}
\usepackage{dynkin-diagrams}
\usepackage{multirow}
\usepackage{multicol}
\usepackage{longtable}
\usepackage{makecell}

\usepackage{float} 
\newtheorem{theorem}{Theorem}[section]
\newtheorem{corollary}[theorem]{Corollary}
\newtheorem{lemma}[theorem]{Lemma}
\newtheorem{proposition}[theorem]{Proposition}

\theoremstyle{definition}
\newtheorem{definition}[theorem]{Definition}
\newtheorem{example}[theorem]{Example}

\newtheorem{remark}[theorem]{Remark}

\numberwithin{equation}{section}

\DeclareMathOperator{\oh}{\mathcal{O}}
\author{Xinyi Fang}
\address{Xinyi Fang, Department of Mathematics, Shanghai Normal University, Shanghai, 200234, PR China}
\thanks{1.Xinyi Fang is supported by National Natural Science Foundation of China (Grant No. 12471040)}
\email{xinyif@shnu.edu.cn}
\author{Duo Li}
\address{Duo Li, Sun-Yat Sen University, School of Mathematics(Zhuhai), Zhuhai, Guangdong, 519082, PR China}
\thanks{2.Duo Li is supported by National Natural Science Foundation of China (Grant No. 12001547)}
\email{liduo5@mail.sysu.edu.cn}
\author{Yanjie Li}
\address{Yanjie Li, AMSS, Chinese Academy of Sciences, 55 ZhongGuanCun East Road, Beijing, 100190, China and  University of Chinese Academy of Sciences, Beijing, China}
\email{liyanjie241@mails.ucas.ac.cn}

\title{ uniform bundles on quadrics}

\begin{document}
\begin{abstract}
We show that there exist only constant morphisms from $\mathbb{Q}^{2n+1}(n\geq 1)$ to $\mathbb{G}(l,2n+1)$ if $l$ is even $(0<l<2n)$ and $(l,2n+1)$ is not $ (2,5)$. As an application, we prove on $\mathbb{Q}^{2m+1}$ and $\mathbb{Q}^{2m+2}(m\geq 3)$, any uniform bundle of rank at most $2m$ splits, which improves the upper bound of splitting for uniform bundles obtained by Kachi and Sato \cite{KS}. We classify all unsplit uniform bundles of minimal rank on $B_n/P_k$ $(k=\frac{2n}{3},k\ge6)$ and $D_n/P_k$ $(k=\frac{2n-2}{3},k\ge 6)$. We partially answer a conjecture of Ellia, which predicts that some uniform bundles of special splitting types on $\mathbb{P}^n$ necessarily split and we find some restrictions on the splitting types of unsplit uniform bundles of minimal rank.
\end{abstract}
\maketitle


\section{Introduction}
In this article, we assume all the varieties are defined over $\mathbb{C}$. We assume that $X$ is a rational homogeneous variety of Picard number $1$ and we call
$X$ a generalized Grassmannian for short. It is well known that $X$ is swept by lines.  Let $E$ be a vector bundle on $X$. We consider  its restriction  $E|_{L}(\simeq \mathcal O_L(a_1(L)\oplus\cdots\oplus\mathcal O_L(a_r(L)))$ to any line $L\subseteq X$. If the splitting type $ (a_1(L),\dots,a_r(L)) $ of   $E|_{L}$ is independent of the choice of $L$,  $E$ is  called a \emph{uniform bundle}. We say a vector bundle \emph{splits} if it can be decomposed as a direct sum of line bundles. We say a vector bundle does not split or a vector bundle is \emph{unsplit} if it cannot be decomposed as a direct sum of line bundles. \\

In \cite{Gro},  Grothendieck shows that every vector bundle on a projective line splits. However, for higher dimensional projective spaces, the splitting of vector bundles is far more intricate.  In \cite{Van}, Van de Ven studies the splitting property  of uniform bundles. He
demonstrates that every uniform bundle of rank $2$ on  $\mathbb{P}^n$ $(n\ge 3)$ splits  and every uniform bundle of rank $2$ on  $\mathbb{P}^2$ is isomorphic to $\mathcal{O}_{\mathbb{P}^2}(a)\oplus \mathcal{O}_{\mathbb{P}^2}(b)$ or $T_{\mathbb{P}^2}(a)$ for some integers $a$ and $b$.  The subsequent work by   Sato \cite{sato1976uniform} and Elencwajg, Hirschowitz, Schneider \cite{EHS} extends these results:
 every uniform bundle on $\mathbb{P}^n$ of  rank  smaller than $n$ splits, while every unsplit uniform bundle  on $\mathbb{P}^n$ of rank $n$ is isomorphic to  $T_{\mathbb{P}^n}(a)$ or $\Omega_{\mathbb{P}^n}(b)$ for some integers $a$ and $b$. \\
 
 Motivated by these advances, we address two central  problems for uniform bundles on a generalized Grassmannian $X$.
 \begin{itemize}
   
     \item \textbf{Problem $1$:} Determine the  splitting threshold $\mu(X) $ such that any uniform bundle of rank at most $\mu(X)$ splits and there exists an unsplit uniform vector bundle of rank $\mu(X)+1$.
     \item \textbf{Problem $2$:}  Classify all uniform bundles of rank $\mu(X)+1$.
   
 \end{itemize}

\quad

Let  $\mathbb{G}(k-1,n-1)$ be the Grassmannian of $k$-dimensional subspaces of $\mathbb{C}^n$. In \cite{guyot1985caracterisation}, Guyot solves the above problems  for  Grassmannians $X=\mathbb{G}(k-1,n-1)(k\leq n-k)$.
%( here $\mathbb{G}(k-1,n)$ parametrizes all $k$-dimensional subspaces of a $(n+1)$-dimensional  vector space). 
For orthogonal Grassmannians $OG(n,2n+1)$,    Mu\~{n}oz, Occhetta and Sol\'{a} Conde \cite{munoz2012uniform} achieve similar results.  \\

Let $\mathbb{Q}^n$ be a projective smooth quadric of dimension $n$, with $n\geq 5$ odd (resp. even). Kachi and Sato \cite[Theorem 4.1]{KS} prove that any uniform vector bundle on $\mathbb{Q}^n$ of rank at most $n-2$ (resp. $n-3$) splits. Mu\~{n}oz, Occhetta and Sol\'{a} Conde provide an alternative proof via a general splitting criterion for low-rank uniform bundles  on varieties covered by lines (see \cite[Corollary 3.3]{munoz2012uniform}). However we will demonstrate their upper bounds are not optimal. Actually, we show that  for $\mathbb{Q}^{2n+1}$ and $\mathbb{Q}^{2n+2}$ $ (n \geq 3)$, every uniform bundle of rank $2n$ splits. \\

In our recent work \cite{FLL24}, we systematically address \textbf{Problem $1$} and \textbf{Problem $2$} for all generalized Grassmannians. A key insight is establishing  a connection between the geometry of VMRT  (variety of minimal rational tangents)  and the splitting behavior of uniform bundles. Specifically, we demonstrate that for an arbitrary generalized Grassmannian $X$,  a uniform bundle $E$ of rank  at most the so-called  $e.d.(\mathrm{VMRT})$ (for the definition of  $e.d.(\mathrm{VMRT})$, see \cite[Section $3$]{FLL24}) must split. Furthermore, for the majority of generalized Grassmannians, the bound $e.d.(\mathrm{VMRT})$ coincides with $\mu(X)$. Our approach diverges from Guyot's algebraic strategy by adopting a more geometric and direct methodology. Notably,  for most generalized Grassmannians whose VMRTs are products of several irreducible varieties, we reduce  \textbf{Problem $2$} to the classification of uniform bundles on projective spaces and quadrics, see \cite[Proposition $4.4$]{FLL24}. This reduction highlights the pivotal role of quadrics in resolving the problem.\\

 Given a generalized Grassmannian $X$, assume  $\mu(X)$ is $e.d.(\text{VMRT})$. Let $E$ be an unsplit uniform bundle of minimal rank on $X$, that is, the rank of $E$ is $\mu(X)+1$. In \cite[Proposition 3.8]{FLL24}, we show that, up to twisting by a suitable line bundle, the splitting type of $E$ is $(1,\dots,1,0,\dots,0)$. Inspired by this result, when $\mu(X)$ is not necessarily $e.d.(\text{VMRT})$,  we show that there are some restrictions on the splitting type of $E$. To be concrete,  if the splitting type of $E$ is $(a_1,a_2,\dots,a_r)$ $(a_1\ge a_2\ge\cdots\ge a_r)$, then $ max\{a_i-a_{i+1}|1\le i\le r-1\}$ is $1$. By using the same strategy, we give an affirmative partial answer to a conjecture of Ellia (see \cite[Page 29, Conjecture]{Ellia}), which predicts that every uniform bundle on $\mathbb{P}^{n}$ of splitting type
 $(\underbrace{a_1,\dots,a_1}_{l_1},\underbrace{a_2,\dots,a_2}_{l_2},\dots, \underbrace{a_k,\dots,a_k}_{l_k})$ with $a_i>a_{i+1}$ and $l_i\le n-1$ for any  $1\le i\le k$ necessarily splits. Motivated by Ellia's conjecture, we propose a similar conjecture for any generalized Grassmannian as follows:\\
 
 \textbf{Conjecture:} For an arbitrary generalized Grassmannian $X$, there exists a maximal positive integer $\nu(X)$ such that every uniform bundle of splitting type \[(\underbrace{a_1,\dots,a_1}_{l_1},\underbrace{a_2,\dots,a_2}_{l_2},\dots, \underbrace{a_k,\dots,a_k}_{l_k})\] with $a_i> a_{i+1}$ and $l_i\le \nu(X)$ for any $1\le i\le k$ necessarily splits.\\
 
In this article, we establish: \\ 

\textbf{Main Results:}  \begin{enumerate}
    \item 
 For $\mathbb{Q}^{2n+1}$ and $\mathbb{Q}^{2n+2}$ $ (n \geq 3)$, every uniform bundle of rank $2n$ splits, with $ \mu(\mathbb{Q}^{2n+1}) = 2n $.  
\item  For $ B_n/P_k$ $ (k = \frac{2n}{3}) $ and $ D_n/P_k $$( k = \frac{2n-2}{3})$, we classify all unsplit uniform bundles of minimal rank.  
\item Give an affirmative partial answer to a conjecture of Ellia in a more general setting and find some restrictions on the splitting types of unsplit uniform bundles of minimal rank.
\end{enumerate} 
\quad\\

Note that for $\mathbb{Q}^{2n+1}$, $e.d.(\mathrm{VMRT})$ is $2n-1$ and the number $\mu(\mathbb{Q}^{2n+1})$ is $2n$.
  So this is the \textbf{first known example} such that the optimal upper bound $\mu(X)$ is bigger than $e.d.(\mathrm{VMRT})$.\\
  
 This article is structured as follows: 
 \begin{enumerate}
     \item 
 
 Section $2$: Analyze morphisms from \( \mathbb{Q}^{2n+1} \) to $\mathbb{G}(l,2n+1)$  and the main result of this section is Proposition \ref{mor from Q to Gr Constant}.  
\item  Theorem \ref{split of rk 2n bd on Q}: Prove splitting theorems for quadrics via relative Harder-Narasimhan filtrations and approximate solutions.
\item Corollaries $3.13–3.16$: Extend classifications to generalized Grassmannians \( B_n/P_k \) and \( D_n/P_k \).  
\item Section $4$: Prove some uniform bundles of special splitting types necessarily split.
\end{enumerate}

\section{Morphisms from quadrics to Grassmannians}
\begin{proposition}\label{mor from Q to Gr Constant}
    There exist only constant morphisms from $\mathbb{Q}^{2n+1}(n\geq 1)$ to $\mathbb{G}(l,2n+1)$ if $l$ is even $(0<l<2n)$ and $(l,2n+1)$ is not $ (2,5)$.
\end{proposition}
\begin{proof}
   We  follow the proof of the main theorem in \cite{tango1976morphisms}. Let $H$ denote the cohomology class of a hyperplane on $\mathbb{Q}^{2n+1}$. The cohomology ring of $\mathbb{Q}^{2n+1}$ is
    \[H^{\bullet}(\mathbb{Q}^{2n+1},\mathbb{Z})=\mathbb{Z}\oplus \mathbb{Z}H \oplus \cdots \oplus \mathbb{Z}H^{n-1}\oplus \mathbb{Z}\frac{H^n}{2}\oplus \cdots \oplus \mathbb{Z}\frac{H^{2n+1}}{2}.\]
    Let $U$ (resp. $Q$) be the universal subbundle (resp. quotient bundle) on $\mathbb{G}(l,2n+1)$. Suppose that $f:\mathbb{Q}^{2n+1}\rightarrow \mathbb{G}(l,2n+1)$ is a non-constant morphism. Let $c_i$ and $d_j$ be rational numbers satisfying
    \[c_i(f^*U^{\vee})=c_iH^i \text{ for }1\leq i \leq l+1 ~\text{  and  }~c_j(f^*Q)=d_jH^j \text{ for }1\leq j \leq 2n+1-l.\]
    By \cite[Proposition 2.1 (i)]{tango1974n}, all coefficients $c_i(1\leq i \leq l+1)$ and $d_j(1\leq j \leq 2n+1-l)$ are non-negative. We note that for any $1\le i,j<n$, $c_i$ and $d_j$ are integers. For any $i,j\geq n$, $2c_i$ and $2d_j$ are integers. %We have
    Then from the exact sequence $0\rightarrow f^*U \rightarrow \mathcal{O}_{\mathbb{Q}^{2n+1}}^{\oplus 2n+2} \rightarrow f^*Q \rightarrow 0$, we get the equality of polynomials:
    \begin{align}\label{eq for chern}
        &(1-c_1t+c_2t^2+\cdots+(-1)^{l+1}c_{l+1}t^{l+1})(1+d_1t+\cdots+d_{2n+1-l}t^{2n+1-l})\\ 
        =&1+(-1)^{l+1}c_{l+1}d_{2n+1-l}t^{2n+2}.\nonumber
    \end{align}
    If $c_{l+1}d_{2n+1-l}$ is $0$, then  from (\ref{eq for chern}), we obtain both $c_1$ and $d_1$ are zero by induction, which implies that $f$ is constant. So we may assume the numbers $c_1,d_1,c_{l+1}$ and $d_{2n+1-l}$ are non-zero. By \cite[Proposition 2.1 (ii)]{tango1974n}, all the rational numbers $c_1,c_2,\dots,c_{l+1}$ and $d_1,d_2,\dots,d_{2n+1-l}$ are positive.\\

    Let $a$ be $\sqrt[2n+2]{c_{l+1}d_{2n+1-l}}$. We set $C_i:=\frac{c_i}{a^i}(1\leq i \leq l+1)$ and $D_j:=\frac{d_j}{a^j}(1\leq j \leq 2n+1-l)$. We note that if $a,C_i$ and $D_j$ are all positive integers, then by the same proof of \cite[Page 204, Case (ii)]{tango1976morphisms}, we can get a contradiction (for details, see Appendix \ref{appA}, Claim). Then it suffices to show that $a,C_i$ and $D_j$ are integers.\\
    
    We first show that $a$ is an integer. Since $l$ is even, similar to the proof in \cite[Lemma 3.3 (ii) and Case 1 of (iii)]{tango1974n}, we can show that $a$ is $\frac{c_{m+1}}{c_m}$ (see Appendix \ref{appA}, Corollary \ref{a is rational}), where $m$ is $\frac{l}{2}$. So $a$ is rational, we may assume $a=s/t$, where $s$ and $t$ are coprime positive integers. By definition, we have $a^{2n+2}=c_{l+1}d_{2n+1-l}$. Since $2c_{l+1}$ and $2d_{2n+1-l}$ are integers, $4a^{2n+2}$ is an integer. So $t^{2n+2}$ divides $4$, which implies that $t$ is $1$, as $n$ is at least $1$. Hence $a$ is an integer.  \\  %If $2$ divides $t$, then $2^{2n+2}$ divides $4$, contradicting $n\geq 1$. So $t$ is $1$, and $a$ is an integer.\\

    We now show that $C_i(=\frac{c_i}{a^i})$ and $D_j(=\frac{d_j}{a^j})$ are integers. Let $F_1(x)\cdots F_k(x)$ be the irreducible factorization of $1-x^{2n+2}$ over $\mathbb{Z}[x]$ with $F_l(0)=1(1\leq l \leq k)$. Then $F_l(x)$ is also irreducible over $\mathbb{Q}[x]$ by Gauss's Lemma. Then
    \begin{align}\label{factor}
        &(1-c_1t+c_2t^2+\cdots+(-1)^{l+1}c_{l+1}t^{l+1})(1+d_1t+\cdots+d_{2n+1-l}t^{2n+1-l})\nonumber\\ 
        =&1+(-1)^{l+1}c_{l+1}d_{2n+1-l}t^{2n+2}=1-a^{2n+2}t^{2n+2}=F_1(at)\cdots F_k(at).\nonumber
    \end{align}
    Note that $F_i(at)$ is also irreducible over $\mathbb{Q}[t]$. We have
    \begin{align}
        &1-c_1t+c_2t^2+\cdots+(-1)^{l+1}c_{l+1}t^{l+1}=F_{i_1}(at)\cdots F_{i_{k_1}}(at)~\text{and}~\nonumber\\
        &1+d_1t+\cdots+d_{2n+1-l}t^{2n+1-l}=F_{j_1}(at)\cdots F_{j_{k_2}}(at).\nonumber
    \end{align}
    Since the coefficients of $F_l(x)$ are integers,  $C_i(=\frac{c_i}{a^i})$ and $D_j(=\frac{d_j}{a^j})$ are integers.
\end{proof}
\section{Uniform bundles on $\mathbb{Q}^{2n+1}(n\geq 3)$}
We are going to use the method in \cite{EHS} to show that any uniform bundle of rank $2n$ on $\mathbb{Q}^{2n+1}(n\geq 3)$ splits. We first fix some notations.\\

Let $E$ be a uniform bundle on $X(=\mathbb{Q}^{2n+1})(n\geq 3)$ of rank $2n$. Assume that $E$ does not split, by \cite[Proposition $3.7$]{FLL24},  we may assume that the splitting type of $E$ is 
\[(\underbrace{0,\dots,0}_{l+1},\underbrace{-1,\dots,-1}_{2n-l-1})\quad (l+1\geq 2n-l-1 \geq 1).
    \]
We denote the moduli of lines on $X$ and the corresponding universal family by:
 \[
    \begin{tikzcd}
        \mathcal{U}(=B_{n+1}/P_{1,2})\arrow[r,"q"]\arrow[d,"p"]& \mathcal{M}(=B_{n+1}/P_2)\\
        X(=B_{n+1}/P_1).
    \end{tikzcd}
 \]
 where $\mathcal{M}$ is the moduli of lines and $\mathcal{U}$ is the universal family.
 The relative Harder-Narasimhan (H-N) filtration of $p^*E$ induces an exact sequence:
 \begin{align}\label{relative HN fli}
     0\rightarrow E_1(=q^*G_1) \rightarrow p^*E \rightarrow E_2(=q^*G_2\otimes p^*\oh_X(-1)) \rightarrow 0,
 \end{align}
 where $G_1$ (resp. $G_2$) is a vector bundle on $\mathcal{M}$ of rank $l+1$ (resp. $2n-l-1$). For each $x\in X$, the restriction of relative H-N filtration to $p^{-1}(x)$ induces a morphism 
 \begin{align}\label{varphi}\psi_x:p^{-1}(x)(\cong \mathbb{Q}^{2n-1})\rightarrow Gr(l+1,2n)(\cong \mathbb{G}(l,2n-1)).\end{align}
 
By \cite[Lemma $2.1$]{FLL24}, we have the following description of cohomology rings:
\begin{align*}
    &H^{\bullet}(X,\mathbb{Q})=\mathbb{Q}[X_1]/(X_1^{2n+2}),\nonumber\\
    &H^{\bullet}(\mathcal{U},\mathbb{Q})=\mathbb{Q}[X_1,X_2]/(\Sigma_{i}(X_1^2,X_2^2)_{n\leq i \leq n+1}),\nonumber\\
    &H^{\bullet}(\mathcal{M},\mathbb{Q})=\mathbb{Q}[X_1+X_2,X_1X_2]/(\Sigma_{i}(X_1^2,X_2^2)_{n\leq i \leq n+1})\nonumber.
\end{align*}
For a  bundle $F$ of rank $r$, the Chern polynomial of $F$ is defined as
$$C_{F}(T):=T^r-c_1(F)T^{r-1}+\cdots+(-1)^rc_r(F).$$
Let $E(T,X_1)=\sum_{k=0}^{2n}e_kX_1^kT^{2n-k}(\in \mathbb{Q}[X_1,T])$ and $S_{i}(T,X_1,X_2)(\in \mathbb{Q}[X_1+X_2,X_1X_2,T])(i=1,2)$ be homogeneous polynomials representing $C_{p^*E}(T)$ and $C_{q^*G_i}(T)$ in the cohomology rings respectively. There are equations
\begin{align*}
    E(T,X_1)=C_{p^*E}(T)~ \text{and}~S_i(T,X_1,X_2)=C_{q^*G_i}(T)(i=1,2)\nonumber.
\end{align*}
Let $R(X_1,X_2)$ be the polynomial $\Sigma_n(X_1^2,X_2^2)$. By (\ref{relative HN fli}), we have an equation of Chern polynomials:
\begin{align}\label{the eq to solve}
    E(T,X_1)-aR(X_1,X_2)=S_1(T,X_1,X_2)S_2(T+X_1,X_1,X_2)\tag{*}.
\end{align}
If $a$ is $0$, then both $S_1(T,X_1,X_2)$ and $S_2(T+X_1,X_1,X_2)$ are polynomials only in variables $T$ and $X_1$. Since $S_i(T,X_1,X_2)(i=1,2)$ are symmetric in $X_1$ and $X_2$, we must have $S_1(T,X_1,X_2)=T^{l+1}$ and $S_2(T,X_1,X_2)=T^{2n-l-1}$. So $c_1(E_1)$ and $c_1(E_2)$ are $0$, and $\psi_x$ is constant for each $x\in X$, which implies that $E$ splits (see, for example, \cite[Proposition $3.5$]{FLL24}). Therefore, we have $a\ne 0$. 
\subsection{Approximate solutions}
To solve the equation (\ref{the eq to solve}), we use the concept of approximate solutions introduced in \cite[Section 5]{EHS}. The following definitions and propositions are basically from \cite[Section 5]{EHS} and the proofs are similar.

\begin{definition}\label{app solu}
    A non-zero homogeneous polynomial $P(T,X_1)$ with rational coefficients in variables $T$ and $X_1$ of degree $2n$ is called an approximate solution if $P(T,X_1)-R(X_1,X_2)$ has a proper divisor $S(T,X_1,X_2)$ which is symmetric in $X_1$ and $X_2$. We call such a divisor a symmetric divisor. 
\end{definition}

\begin{example}\label{app solu ass *}
    By  the equation (\ref{the eq to solve}), both $\frac{1}{a}E(T,X_1)$ and $\frac{1}{a}E(T-X_1,X_1)$ are approximate solutions. We call them approximate solutions associated with (\ref{the eq to solve}).
\end{example}
%We are going to classify approximate solutions associated with (\ref{the eq to solve}).
In the following lemma, there are some restrictions on the coefficient of $X_1^{2n}$ for an arbitrary approximate solution $P(T,X_1)$. 
\begin{lemma}\label{Top coeff of app solu}
    Let $P(T,X_1)=\sum_{k=0}^{2n}p_kX_1^kT^{2n-k}$ be an approximate solution. Then one of the following holds. \\
    $(1)$ Any symmetric divisor of $P(T,X_1)-R(X_1,X_2)$ is of degree one.\\
    $(2)$ The coefficient $p_{2n}$ is $0$ and the zero set of $P(0,1)-R(1,z)$ is $(\{z|z^{2n+2}-1=0\})\backslash \{1,-1\}$.\\
    $(3)$ The coefficient $p_{2n}$ is $1$ and the zero set of $P(0,1)-R(1,z)$ is $(\{z|z^{2n}-1=0\}\cup \{0\})\backslash \{1,-1\}$.
\end{lemma}
\begin{proof}
    Let $S(T,X_1,X_2)$ be a symmetric divisor of $P(T,X_1)-R(X_1,X_2)$. Then
    $S(0,X_1,X_2)$ divides $ p_{2n}X_1^{2n}-R(X_1,X_2) $. As $S(T,X_1,X_2)$ is symmetric in $X_1$ and $X_2$, $S(0,X_1,X_2) $ divides $ p_{2n}X_2^{2n}-R(X_1,X_2). $
    Therefore $S(0,X_1,X_2)$ divides $ p_{2n}(X_1^{2n}-X_2^{2n})$.
    %If $p_{2n}$ is not $0$, we have $S(0,X_1,X_2)\mid X_1^{2n}-X_2^{2n}$. 
    By the equation $R(X_1,X_2)(X_1^2-X_2^2)=X_1^{2n+2}-X_2^{2n+2}$, we have
    \begin{align*}
        (X_1^2-X_2^2)(p_{2n}X_1^{2n}-R(X_1,X_2))+X_2^2(X_1^{2n}-X_2^{2n})=(p_{2n}-1)X_1^{2n}(X_1^2-X_2^2).
    \end{align*}

    If $p_{2n}$ is neither $0$ nor $1$, we have $S(0,X_1,X_2)\mid X_1^{2n}-X_2^{2n}$ and $S(0,X_1,X_2)\mid X_1^{2n}(X_1-X_2)(X_1+X_2)$. Since $S(0,X_1,X_2)$ is symmetric in $X_1$ and $X_2$, $S(0,X_1,X_2)$ is $c(X_1+X_2)$  for some $c\in \mathbb{Q}$. In particular,  $\deg(S)$ is $1$.\\

    If $p_{2n}$ is $0$, then $P(0,1)-R(1,z)$ is $-\frac{z^{2n+2}-1}{z^2-1}(=-(z^{2n}+z^{2n-2}+\cdots+1))$. If $p_{2n}$ is $1$, then $P(0,1)-R(1,z)$ is $-z^2\frac{z^{2n}-1}{z^2-1}=(-(z^{2n}+z^{2n-2}+\cdots+z^2))$. Then the assertions follow immediately.
\iffalse
    If $p_{2n}$ is $0$. Let $F(T,X_1,X_2)$ be the polynomial satisfying
    \begin{align}\label{eq when p2n zero}
        P(T,X_1)-R(X_1,X_2)=S(T,X_1,X_2)F(T,X_1,X_2).
    \end{align}
    Expand $S(T,X_1,X_2)$ and $F(T,X_1,X_2)$ as polynomials in the variable $T$, we get
    \begin{align*}
        &S(T,X_1,X_2)=s_mT^{m}+\cdots+s_1(X_1,X_2)T+s_{0}(X_1,X_2)\\
        &F(T,X_1,X_2)=f_{2n-m}T^{2n-m}+\cdots+f_1(X_1,X_2)T+f_{0}(X_1,X_2).
    \end{align*}
    By substituting $T=0$ in (\ref{eq when p2n zero}), we have $s_{0}(X_1,X_2)f_{0}(X_1,X_2)=-R(X_1,X_2)$. Since both $s_0(X_1,X_2)$ and $R(X_1,X_2)$ are symmetric in $X_1,X_2$, $f_0(X_1,X_2)$ is also symmetric in $X_1,X_2$. As $z^{2n+2}-1=0$ has no multiple roots, we also have $(s_0,f_0)=1$. Consider the coefficients of $T$ in both sides of (\ref{eq when p2n zero}), we have $s_1(X_1,X_2)f_0(X_1,X_2)+s_0(X_1,X_2)f_1(X_1,X_2)=p_{2n-1}X_1^{2n-1}.$ If $p_{2n-1}$ is $0$, we have $s_{0}\mid s_1f_0$. From $(s_0,f_0)=1$, we get $s_0\mid s_1$
\fi
\end{proof}
%Note that if $p_{2n}\in \{0,1\}$, then the solutions of $p(0,1)-R(1,z)=0$ is a subset of 
\begin{definition}\label{prim app solu}
    We call an approximate solution $P(T,X_1)=\sum_{k=0}^{2n}p_kX_1^kT^{2n-k}$ a primitive approximate solution if $p_{2n}\in \{0,1\}$ and $P(T,X_1)-R(X_1,X_2)$ has a symmetric divisor $S_0(T,X_1,X_2)$ such that there is a  $2(n-p_{2n}+1)$-th primitive unit root $y_0$ satisfying $S_0(0,1,y_0)=0$.
\end{definition}
As in \cite{EHS}, we have the following classifications of primitive approximate solutions.
\begin{proposition}\label{classification for prim solu}
    Let $P(T,X_1)$ be a primitive approximate solutions. If $p_{2n}$ is $0$, we have $P(T,X_1)=bT^{2n}$ for some $b\in \mathbb{Q}$. If $p_{2n}$ is $1$, we have $P(T,X_1)=\Sigma_{n}(bT^2,X_1^2)$ for some $b\in \mathbb{Q}$.
\end{proposition}
The proof of Proposition \ref{classification for prim solu} is similar to that of \cite[Proposition 6.1 and Proposition 6.2]{EHS}, we leave them in the Appendix (see Propositions \ref{A1} and \ref{A2}).
\subsection{The case $l$ is even}
We first  show that  $l$ is not  even and we will prove it by contradiction. Now suppose that $l$ is even. If $l$ is smaller than $2n-2$ and $(l,2n-1)$ is not $(2,5)$, the morphism $\psi_x$ (for the definition of $\psi_x$, see (\ref{varphi})) is constant for each $x\in X$ according to Proposition \ref{mor from Q to Gr Constant}. Then $E$ splits. We exclude the remaining cases $l=2n-2$ and $(l,2n-1)=(2,5)$ by calculations.
\begin{proposition}\label{exclude by cal l=2n-2}
    There does not exist an unsplit uniform bundle of rank $2n$ whose  splitting type is $(0,\dots,0,-1)$ on $\mathbb{Q}^{2n+1}(n\geq 3)$.
\end{proposition}
\begin{proof}
    Suppose $E$ is unsplit, by the same calculation as in \cite[Theorem 4.3, Case I]{FLL24}, we have $c_1(E_2)=(X_1+X_2)-X_1=X_2$. Let $f(X_1,X_2)$ be a homogeneous polynomial of degree $2n-1$ which is  symmetric in $X_1$ and $X_2$ and represents $c_{2n-1}(E_1)$. By comparing the coefficients of $T^0$ on the left and right sides of the equation (\ref{the eq to solve}), we get
    \[f(X_1,X_2)X_2=e_{2n}X_1^{2n}-a(X_1^{2n}+X_1^{2n-2}X_2^2+\cdots+X_2^{2n}).\]
    Then we must have $e_{2n}=a$ and $f(X_1,X_2)=-a(X_1^{2n-2}X_2+X_1^{2n-4}X_2^{3}+\cdots+X_2^{2n-1})$, contradicting the assumption that $f$ is symmetric in $X_1$ and $X_2$.
\end{proof}
We now exclude the case $(l,2n-1)=(2,5)$.
\begin{proposition}\label{exclude by cal (2,5)}
    There does not exist an unsplit uniform of rank $6$ whose  splitting type is $(0,0,0,-1,-1,-1)$ on $\mathbb{Q}^7$.
\end{proposition}
\begin{proof}
    Suppose $E$ is unsplit. Then $a$ in the equation (\ref{the eq to solve}) is not  $0$ and 
    $\frac{1}{a}E(T,X_1)=\frac{1}{a}\sum_{k=0}^{6}e_kX_1^kT^{6-k}$ is an approximate solution which has a symmetric divisor of degree $3$. By Lemma \ref{Top coeff of app solu}, we have $\frac{1}{a}e_{6}=0$ or $\frac{1}{a}e_{6}=1$. 
    Let $f(X_1,X_2)$ be a homogeneous polynomial representing $c_3(E_1)$.\\
    
    If $\frac{1}{a}e_{6}$ is $0$, we have $f(X_1,X_2)\mid R(X_1,X_2)(=X_1^6+X_1^4X_2^2+X_1^2X_2^4+X_2^6)$. Since 
    %the prime factorization of $X^6+X^4+X^2+1$ over $\mathbb{Q}[X]$ is $(X^4+1)(X^2+1)$ (see \textcolor{red}{???}), 
    the prime factorization of $X_1^6+X_1^4X_2^2+X_1^2X_2^4+X_2^6$ over $\mathbb{Q}[X_1,X_2]$ is $(X_1^4+X_2^4)(X_1^2+X_2^2)$, $R(X_1,X_2)$ has no divisor symmetric in $X_1$ and $X_2$ of degree $3$. \\

    If $\frac{1}{a}e_{6}$ is $1$, then $f(X_1,X_2)$ divides $X_1^{6}-R(X_1,X_2)(=-X_2^2(X_1^4+X_1^2X_2^2+X_2^4))$. The prime factorization of $X_1^4+X_1^2X_2^2+X_2^4$ over $\mathbb{Q}[X_1,X_2]$ is $(X_1^2+X_1X_2+X_2^2)(X_1^2-X_1X_2+X_2^2)$. Therefore, $X_1^{6}-R(X_1,X_2)$ has no divisor symmetric in $X_1$ and $X_2$ of degree $3$.\\

    In both cases, we get contradictions.
\end{proof}

\subsection{The case $l$ is odd}
 Suppose that $l$ is odd. We begin with a lemma.
%The key is an analogue of \cite[Proposition 5.4]{EHS} (Proposition \ref{ass app solu is prim}), 

\begin{lemma}\label{no real root of E-aR}
    When $l$ is odd, the equation $E(t,1)-aR(1,0)=0$ has no roots in $\mathbb{R}$.
\end{lemma}
\begin{proof}
    In the equation (\ref{the eq to solve}): $E(T,X_1)-aR(X_1,X_2)=S_1(T,X_1,X_2)S_2(T+X_1,X_1,X_2)$, we let  $X_1$ be $0$ and let $X_2$ be $1$. Then we get an equation $T^{2n}-a=S_1(T,0,1)S_2(T,0,1)$. We write $S_1(T,0,X_2)$  and $S_2(T,0,X_2)$   as follows:   $S_1(T,0,X_2)=\sum_{i=0}^{l+1}a_iX_2^iT^{l+1-i}$ and $S_2(T,0,X_2)=\sum_{j=0}^{2n-1-l}b_jX_2^jT^{2n-1-l-j}$, where $a_i(0\leq i\leq l+1)$ and $b_j(0\leq j \leq 2n-1-l)$ are rational numbers. Then $-a$ is $a_{l+1}b_{2n-1-l}$. We now wish to show  $-a>0$. To this end, for any $x$ in $ X$, we consider the embedding $i_x:p^{-1}(x)(=\mathbb{Q}^{2n-1})\hookrightarrow \mathcal{U}(=B_{n+1}/P_{1,2})$, which induces a morphism:
$$i_x^*:H^{\bullet}(\mathcal{U},\mathbb{Q})\cong \mathbb{Q}[X_1,X_2]/(\Sigma_{n}(X_1^2,X_2^2),\Sigma_{n+1}(X_1^2,X_2^2))\rightarrow \mathbb{Q}[X_2]/X_2^{2n}\cong H^{\bullet}(\mathbb{Q}^{2n-1},\mathbb{Q}).$$ Under the above identifications, we have $S_1(T,0,X_2)=C_{E_1|_{p^{-1}(x)}}(T)=C_{\psi_x^*U}(T)$ and $S_2(T,0,X_2)=C_{E_2|_{p^{-1}(x)}}(T)=C_{\psi_x^*Q}(T)$, where $U$ (resp, $Q$) is the universal subbundle (resp. quotient bundle) on $\mathbb{G}(l,2n-1)$ and $\psi_x$ is the morphism as in (\ref{varphi}). So we have
    \begin{align*}
        &a_{l+1}X_2^{l+1}=(-1)^{l+1}c_{l+1}(\psi_x^*U)=c_{l+1}(\psi_x^*U^{\vee})~\text{and}~\\
        &b_{2n-1-l}X_2^{2n-1-l}=(-1)^{2n-1-l}c_{2n-1-l}(\psi_x^*Q)=c_{2n-1-l}(\psi_x^*Q)~(\text{since $l$ is odd,~$2n-1-l$ is even}).
    \end{align*}
    By \cite[Proposition 2.1]{tango1974n}, $c_{l+1}(\psi_x^*U^{\vee})$ and $c_{2n-1-l}(\psi_x^*Q)$ are numerically non-negative, hence both $a_{l+1}$ and $b_{2n-1-l}$ are non-negative. As $-a$ is $a_{l+1}b_{2n-1-l}$ and $a$ is not $0$, we have $-a>0$. So for any $t\in \mathbb{R}$, $t^{2n}-a$ is bigger than $0$. In other words, $S_1(t,0,1)$ and $S_2(t,0,1)$ are non-zero for any $t\in \mathbb{R}$. By the equations $E(T,1)-aR(1,0)=S_1(T,1,0)S_2(T+1,1,0)=S_1(T,0,1)S_2(T+1,0,1)$, for any $t\in \mathbb{R}$, $E(t,1)-aR(1,0)$ is not $0$.
\end{proof}
\begin{proposition}\label{ass app solu is prim}
    When $l$ is odd and $n$ is at least $3$, the approximate solution $\frac{1}{a}E(T,X_1)$ or $\frac{1}{a}E(T-X_1,X_1)$ associated with (\ref{the eq to solve})  is a primitive approximate solution.
\end{proposition}
\begin{proof}
    Recall the equation $\frac{1}{a}E(t,1)-R(1,z)=\frac{1}{a}S_1(t,1,z)S_2(t+1,1,z)$. Note that $l$ is odd implies $l+1\geq 2n-1-l>1$. So by Lemma \ref{Top coeff of app solu}, we have $\frac{1}{a}e_{2n}=0$ or $\frac{1}{a}e_{2n}=1$. And when $\frac{1}{a}e_{2n}$ is $0$, $\frac{1}{a}E(0,1)-R(1,\exp\frac{2\pi i}{2n+2})$ vanishes; when $\frac{1}{a}e_{2n}$ is $1$, $\frac{1}{a}E(0,1)-R(1,\exp\frac{2\pi i}{2n})$ vanishes.\\

    Suppose that $n$ is $3$. When $\frac{1}{a}e_{2n}$ is $0$, we have $\frac{1}{a}S_1(0,1,z)S_2(1,1,z)=-(z^6+z^4+z^2+1)$. The prime factorization of $z^6+z^4+z^2+1$ over $\mathbb{Q}[z]$ is $(z^4+1)(z^2+1)$. Note $\deg(S_1)\geq \deg(S_2)$ and $\deg(S_2)>1$, we must have $S_1(0,1,z)=\lambda(z^4+1)$ for some $\lambda\in \mathbb{Q}$. Then $\exp\frac{2\pi i}{8}$ is a root of $S_1(0,1,z)$ and hence $\frac{1}{a}E(T,X_1)$ is primitive. When $\frac{1}{a}e_{2n}$ is $1$, we have $\frac{1}{a}S_1(0,1,z)S_2(1,1,z)=-(z^6+z^4+z^2)=-z^2(z^4+z^2+1)=-z^2(z^2+z+1)(z^2-z+1)$. Note that $S_1$ is symmetric in $X_1, X_2$ and  $\deg(S_1)$ is at least $ \deg(S_2)$, we have $S_1(0,1,z)=\lambda(z^4+z^2+1)$. Then $\exp\frac{2\pi i}{6}$ is a root of $S_1(0,1,z)$ and hence $\frac{1}{a}E(T,X_1)$ is also primitive.\\
    
    %as in the proof of Proposition \ref{exclude by cal (2,5)} we have $c_{l+1}(E_1)c_{2n-l-1}(E_2)=-aX_2^2(X_1^2+X_1X_2+X^2)(X_1^2-X_1X_2+X_2^2)$. As $l+1\geq 2n-l1$ and $c_{l+1}(E_1)$ is represented by a polynomial symmetric in $X_1,X_2$, we must have $c_{l+1}(E_1)=\lambda(X_1^2+X_1X_2+X_2^2)(X_1^2-X_1X_2+X_2^2)$ for some $\lambda\in \mathbb{Q}$ and $S_1(0,1,z)=\lambda(z^4+z^2+1)$. Then $\frac{1}{a}E(T,X_1)$ is also primitive.\\

    Suppose now $n$ is at least $4$ and  $\frac{1}{a}E(T,X_1)$ is not a primitive solution. Then we have \begin{align}\label{zero in K}
        S_2(1,1,\exp\frac{2\pi i}{2n+2})=0\text{ or }S_2(1,1,\exp\frac{2\pi i}{2n})=0.
    \end{align} We now show that $\frac{1}{a}E(T-X_1,X_1)$ is a primitive solution. First we have the following inequalities:
    $$\frac{\pi}{2n}<\frac{2\pi}{2n+2}<\frac{2\pi}{2n}<\frac{3\pi}{2n}<\frac{4\pi}{2n+2}.$$
    (Note that for the last inequality, we use the condition $n\geq 4$).
    Since the roots of $S_2(0,1,z)$ satisfy $z^{2n}=1$ or $z^{2n+2}=1$, it suffices to show that $S_2(0,1,z)$ has a non-zero root $y_0$ with argument satisfying $\frac{\pi}{2n}< \arg y_0 <\frac{3\pi}{2n}$. By the above inequalities, we have $y_0=\frac{2\pi i}{2n+2}$ or $y_0=\frac{2\pi i}{2n}$. Then $\frac{1}{a}E(T-X_1,X_1)$ is a primitive approximate solution.\\

    It is enough to show that for any $t\in \mathbb{R}$, $S_2(t+1,1,z)$ as a polynomial of $z$ has a non-zero root with argument in $(\frac{\pi}{2n},\frac{3\pi}{2n})$. Denote by $K$ the set $\{t\in \mathbb{R}~|~\exists ~r(t)(\ne 0)\in \mathbb{C}, ~S_2(t+1,1,r(t))= 0 \text{ with }\frac{\pi}{2n}<\arg r(t) <\frac{3\pi}{2n} \}$. By (\ref{zero in K}), we have $0\in K$. Note that $K$ is open by construction, if we can show $K$ is closed, then $K$ is $\mathbb{R}$. Let $t_0$ be a limit point of $K$. Then $S_2(t_0+1,1,z)$ has a root $r(t_0)$ which is a limit point of $\{r(t)~|~t\in K\}$. By Lemma \ref{no real root of E-aR}, for any $t\in \mathbb{R}$, as a polynomial in $z,$ the roots of $\frac{1}{a}E(t,1)-R(1,z)$ are non-zero. So the roots of $S_2(t+1,1,z)$ are also non-zero. In particular, $r(t_0)$ is not zero. Suppose that $t_0\notin K$, then we have $\arg r(t_0)=\frac{m\pi}{2n}$, where $m$ is $1$ or $3$. We may assume $r(t_0)=\rho \exp\frac{im\pi}{2n}$, where $\rho$ is a positive real number. Then $\frac{1}{a}E(t_0,1)-R(1,\rho \exp\frac{im\pi}{2n})$ is $0$. To get a contradiction, we show for any $d\in \mathbb{R}$, we have $R(1,\rho \exp\frac{im\pi}{2n})+d\ne 0$. If $R(1,\rho \exp\frac{im\pi}{2n})+d=0$, we have the following identities:
    \begin{align*}
        &(R(1,\rho \exp\frac{im\pi}{2n})+d)(\rho^2\exp\frac{2im\pi}{2n}-1)\\
        =&R(1,\rho \exp\frac{im\pi}{2n})(\rho^2\exp\frac{2im\pi}{2n}-1)+d(\rho^2\exp\frac{2im\pi}{2n}-1)\\
        =&\rho^{2n+2}\exp\frac{im(2n+2)\pi}{2n}-1+d\rho^2\exp\frac{2im\pi}{2n}-d\\
        =&(d\rho^2-\rho^{2n+2})\exp\frac{im\pi}{n}-(d+1)=0. %(\text{we use }m=1,3\text{ here})
    \end{align*}
    Since $n$ is bigger than $3$ and $m$ is $1$ or $3$, $\exp\frac{im\pi}{n}$ is not real. We must have $d+1=0$ and $d\rho^2-\rho^{2n+2}=0$. But it implies $d\rho^2-\rho^{2n+2}=-\rho^2-\rho^{2n+2}=0$, which is absurd.
\end{proof}
Now we complete the proof for the case $l$ is odd. By Proposition \ref{classification for prim solu} and Proposition \ref{ass app solu is prim}, we have the following possibilities (note that the coefficient of $T^{2n}$ in $E(T,X_1)$ is $1$):\\

 $E(T,X_1)$ is $T^{2n}$ or $a\Sigma_n(\frac{T^2}{b_1},X_1^2)$; $E(T-X_1,X_1)$ is $T^{2n}$ or $a\Sigma_n(\frac{T^2}{b_2},X_1^2)$,
 where $b_i(\in \mathbb{Q})$ satisfy $b_i^n=a$$(i=1,2)$.
\begin{proposition}\label{exclude l odd}
   The above possibilities are all impossible.
\end{proposition}
\begin{proof}
    If $E(T,X_1)$ is $T^{2n}$,  $T^{2n}-aR(X_1,X_2)$ is $S_1(T,X_1,X_2)S_2(T+X_1,X_1,X_2)$. However, we have
    \[X_1-\exp \frac{2\pi i}{2n+2}X_2 \mid R(X_1,X_2),~(X_1-\exp \frac{2\pi i}{2n+2}X_2)^2 \nmid R(X_1,X_2),~X_1-\exp \frac{2\pi i}{2n+2}X_2 \nmid 1.\]
    By Eisenstein's criterion, $T^{2n}-aR(X_1,X_2)$ is irreducible considered as the polynomial in the variable $T$ with coefficients in $\mathbb{Q}[X_1,X_2]$. This leads to a contradiction. Similar arguments can be applied to the case $E(T-X_1,X_1)$ is $T^{2n}$.\\

    If $E(T,X_1)$ is $b_1^n\Sigma_n(\frac{T^2}{b_1},X_1^2)(=\Sigma_{n}(T^2,b_1X_1^2))$, then we have the equalities $E(T,X_1)-aR(X_1,X_2)=\Sigma_n(T^2,b_1X_1^2)-\Sigma_n(b_1X_1^2,b_1X_2^2)=(T^2-b_1X_2^2)\Sigma_{n-1}(T^2,b_1X_1^2,b_1X_2^2)$ (for the last equality, see \cite[Section 7.2]{EHS} for example). We will make use of the following claim, whose proof will be given in Proposition \ref{A3}.\\
    
 \textbf{Claim:}  
 $\Sigma_{n-1}(T^2,X_1^2,X_2^2)$ is irreducible in $\mathbb{C}[T,X_1,X_2]$.\\
    
    As $b_1$ is not $ 0$, $\Sigma_{n-1}(T^2,b_1X_1^2,b_1X_2^2)$ is irreducible. We have $(T^2-b_1X_2^2)\Sigma_{n-1}(T^2,b_1X_1^2,b_1X_2^2)=S_1(T,X_1,X_2)S_2(T+X_1,X_1,X_2)$. Since $S_1(T,X_1,X_2)$ is symmetric in $X_1,X_2$ and the coefficient of $T^{2n}$ of $S_1$ is $1$,
    then  $S_1(T,X_1,X_2)$ is $\Sigma_{n-1}(T^2,b_1X_1^2,b_1X_2^2)$ and hence $S_2(T+X_1,X_1,X_2)$ is $T^2-b_1X_2^2$. So $S_2(T,X_1,X_2)$ is $(T-X_1)^2-b_1X_2^2,$ which is not symmetric in $X_1,X_2.$ We get a contradiction.\\
    
    If $E(T-X_1,X_1)$ is $b_2^n\Sigma_n(\frac{T^2}{b_2},X_1^2)(=\Sigma_{n}(T^2,b_2X_1^2))$, we have $(T^2-b_2X_2^2)\Sigma_{n-1}(T^2,b_2X_1^2,b_2X_2^2)=S_1(T-X_1,X_1,X_2)S_2(T,X_1,X_2)$. Similarly, $S_2(T,X_1,X_2)$ is $\Sigma_{n-1}(T^2,b_2X_1^2,b_2X_2^2)$. But  $\deg(S_2)$ is at most $ \deg(S_1)$ and $n$ is at least $3$, it is impossible. 
\end{proof}
Now  we prove our main theorem.
\begin{theorem}\label{split of rk 2n bd on Q}
  Assume $n$ is at least $3$, then  every uniform bundle of rank $2n$ on $\mathbb{Q}^{2n+1}$ or $\mathbb{Q}^{2n+2}$ splits and $\mu(\mathbb{Q}^{2n+1})$ is $2n$.
\end{theorem}
\begin{proof}
Combining Proposition \ref{mor from Q to Gr Constant}, Proposition \ref{exclude by cal l=2n-2}, Proposition \ref{exclude by cal (2,5)} and Proposition \ref{exclude l odd},     we can show that every uniform bundle of rank $2n$ on $\mathbb{Q}^{2n+1}$ splits for $n\geq 3$. As the tangent bundle $T_{\mathbb{Q}^{2n+1}}$ is unsplit, the threshold $\mu(\mathbb{Q}^{2n+1})$ is $2n$. \\

Now let $E'$ be a uniform bundle of rank $2n$ on $\mathbb{Q}^{2n+2}$. For every smooth hyperplane section $\mathbb{Q}^{2n+1}\hookrightarrow \mathbb{Q}^{2n+2}$, the restriction $E'|_{\mathbb{Q}^{2n+1}}$ is a uniform bundle of rank $2n$ on $\mathbb{Q}^{2n+1}$, hence $E'|_{\mathbb{Q}^{2n+1}}$ splits. So by \cite[Corollary 3.3]{ottaviani1989some}, $E'$ splits.
\end{proof}

\begin{remark}
For a majority of  generalized Grassmannians $X$, $\mu(X)$ is  $e.d.(\mathrm{VMRT})$ (see \cite[Page 3, Table 2]{FLL24}). Note that  the  $e.d.(\mathrm{VMRT})$ of  both $\mathbb{Q}^{2n+1}$ and $\mathbb{Q}^{2n+2}$ are $2n-1$. Theorem \ref{split of rk 2n bd on Q} shows that the splitting thresholds for uniform vector bundles on $\mathbb{Q}^{2n+1}$ and $\mathbb{Q}^{2n+2}(n\geq 3)$ are at least $2n$. So $\mathbb{Q}^{2n+1}$ and $\mathbb{Q}^{2n+2}(n\geq 3)$ are the first known examples such that $\mu(X)$ is bigger than $e.d.(\mathrm{VMRT})$.
\end{remark}

In \cite{FLL24}, the authors classify all unsplit uniform bundles of minimal rank on the generalized Grassmannians $B_n/P_k~(2\leq k < \frac{2n}{3})$, $B_n/P_{n-2}$, $B_n/P_{n-1}$ and $D_n/P_k~(2\leq k < \frac{2n-2}{3})$, $D_n/P_{n-3}$, $D_n/P_{n-2}$. As direct corollaries of Theorem \ref{split of rk 2n bd on Q}, we can give further classification results for uniform bundles on $B_n/P_k~(\frac{2n}{3}\leq k \leq n-3)$ and $D_n/P_k~(\frac{2n-2}{3}\leq k \leq n-4)$.

\begin{corollary}\label{B1}
Let $X$ be  $B_n/P_k$, where $k$ is $\frac{2n}{3}$ and  is at least $ 6$. Let $E$ be a uniform vector bundle on $X$ of rank $r$.
\begin{itemize}
    \item If $r$ is smaller than $k$, then $E$ is a direct sum of line bundles.
    \item If $r$ is $k$, then $E$ is either a direct sum of line bundles or  $E_{\lambda_1}\otimes L$ or  $E_{\lambda_1}^{\vee}\otimes L$ for some line bundle $L$, where $E_{\lambda_1}$ is the irreducible homogeneous bundle corresponding to the highest weight $\lambda_1$. 
\end{itemize}
\end{corollary}
  \begin{proof}
Note that $e.d.(\mathrm{VMRT})$ of $X$ is $k-1(=2n-2k-1)$.  By
\cite[Theorem 1.1 (1)]{FLL24}, the first assertion follows. For the case $r=k$, since $2n-2k+1=k+1$ is at least $ 7$, $E|_{\mathbb{Q}^{2n-2k+1}}$ splits by Theorem \ref{split of rk 2n bd on Q}. Suppose $E$  is unsplit, then $E|_{\mathbb{P}^k}$ is also unsplit. The second assertion then follows from \cite[Proposition 4.4]{FLL24}. 
  \end{proof}  
  
\begin{corollary}\label{B2}
Let $X$ be  $B_n/P_k$ with $\frac{2n}{3}< k \leq n-3$. Every uniform bundle of rank $2n-2k$ on $X$ splits.
\end{corollary}
\begin{proof}
If $E$ is a uniform bundle of rank $2n-2k$ on $X$, then $E|_{\mathbb{P}^k}$ splits, as $2n-2k$ is smaller than $k$. On the other hand, since $k$ is at most $ n-3$ and hence $2n-2k+1$ is at least $7$, $E|_{\mathbb{Q}^{2n-2k+1}}$ splits by Theorem \ref{split of rk 2n bd on Q}. Because any $2$-plane in $X$ is contained in a $\mathbb{P}^k$ or $\mathbb{Q}^{2n-2k+1}$, $E$ splits by \cite[Corollary 3.6]{du2021vector}.
\end{proof}  

Similar to the proofs of the above corollaries, we can prove the following results.
\begin{corollary}\label{D1}
Let $X$ be  $D_n/P_k$, where $k$ is $\frac{2n-2}{3}$ and $k$ is at least $ 6$. Let $E$ be a uniform vector bundle on $X$ of rank $r$.
\begin{itemize}
    \item If $r$ is smaller than $k$, then $E$ is a direct sum of line bundles.
    \item If $r$ is $k$, then $E$ is either a direct sum of line bundles or  $E_{\lambda_1}\otimes L$ or  $E_{\lambda_1}^{\vee}\otimes L$ for some line bundle $L$, where $E_{\lambda_1}$ is the irreducible homogeneous bundle corresponding to the highest weight $\lambda_1$. 
\end{itemize}
\end{corollary}
  \begin{corollary}\label{D2}
Let $X$ be $D_n/P_k$ with $\frac{2n-2}{3}< k \leq n-4$. Every uniform bundle of rank $2n-2k-2$ on $X$ splits.
\end{corollary}
\section{Splitting type of unsplit uniform bundle of minimal rank}
There are some restrictions on the splitting types of unsplit uniform bundles of minimal rank. The following theorem generalizes \cite[Corollary 4.7]{du2021vector} to generalized Grassmannians associated with short roots.
 \begin{theorem}\label{splitting type of unsplit bundle}
     Assume $E$ is an unsplit uniform bundle on a generalized Grassmannian $X$ whose rank is $\mu(X)+1$. If the splitting type of $E$ is $(a_1,a_2,\dots,a_r)$ $(a_1\ge a_2\ge\cdots\ge a_r)$, then $ max\{a_i-a_{i+1}|1\le i\le r-1\}$ is $1$.
 \end{theorem}

  We now prove Theorem \ref{splitting type of unsplit bundle}.
 \begin{proof}
    Let $L$ be a line in $X$. We denote the inclusion morphism by $f_L:L(\simeq \mathbb{P}^1)\rightarrow X$. As $X$ is a homogeneous variety, the tangent bundle $T_X$ is globally generated. Then $f_L^*(T_X)$ is also globally generated and hence $H^1(L,f_L^*(T_X))$ vanishes. So Mor$(\mathbb{P}^1,X)$ is smooth at $[f_L]$.\\

      If $ max\{a_i-a_{i+1}|1\le i\le r-1\}$ is $0$, by \cite[Theorem 1.2]{Pan2015Tri}, $E$ is trivial. If $ max\{a_i-a_{i+1}|1\le i\le r-1\}$ is at least $2$, there would exist a number $j<r$ such that $a_j-a_{j+1}\ge 2$. By \cite[Proposition 3.1]{PRT24}, there would be a subbundle $W$ of $E$ satisfying the following two properties:
     \begin{itemize}
         \item $W$ is a uniform bundle of splitting type $(a_1,\dots,a_j)$.
         \item the quotient $U\triangleq E/W$ is a uniform bundle of splitting type 
             $(a_{j+1},\dots,a_r)$.
     \end{itemize}
     As rk$(U)$ and rk$(W)$ are at most $\mu(X)$, both $U$ and $W$ split. Since $\text{Ext}^1(U,W)$ vanishes, the exact sequence $0\rightarrow W\rightarrow E\rightarrow U\rightarrow 0$ would split. So $E(\simeq U\oplus W)$ would be a direct sum of line bundles. 
 \end{proof}
 In \cite[Page 29, Conjecture]{Ellia}, Ellia proposes a conjecture that  every uniform bundle on $\mathbb{P}^{n}$ of  splitting type 
 $$(\underbrace{a_1,\dots,a_1}_{l_1},\underbrace{a_2,\dots,a_2}_{l_2},\dots, \underbrace{a_k,\dots,a_k}_{l_k})$$ with $a_i> a_{i+1}$ and $l_i\le n-1$ for any $1\le i\le k$ necessarily splits. Using the same argument as in Theorem \ref{splitting type of unsplit bundle}, we can reduce this conjecture to the case $l_i\le n-1$ and $a_i-a_{i+1}=1$ for any $1\le i\le k-1$. In particular, the following Theorem \ref{splitting type2} partially answers Ellia's conjecture in a more general setting.
 \begin{theorem}\label{splitting type2}
     Given a generalized Grassmannian $X$. Let $E$ be a uniform bundle on $X$. Assume that the splitting type of $E$ is $(\underbrace{a_1,\dots,a_1}_{l_1},\underbrace{a_2,\dots,a_2}_{l_2},\dots, \underbrace{a_k,\dots,a_k}_{l_k})$ with $a_i> a_{i+1}$. If for any $1\le i\le k-1$,  $a_i-a_{i+1} $ is at least $2$, then $E$ splits.
 \end{theorem}

 \begin{proof}
     We prove by induction on $k$. The case $k=1$ is obviously true. Suppose the case $k<n$ is true. We now prove the case $k=n$. \\

    Following the same strategy in the proof of Theorem \ref{splitting type of unsplit bundle}, by \cite[Proposition 3.1]{PRT24}, there exists a subbundle $W$ of $E$ satisfying the following two properties:
    
     \begin{itemize}
         \item $W$ is a uniform bundle of splitting type $(\underbrace{a_1,\dots,a_1}_{l_1})$.
         \item the quotient $U\triangleq E/W$ is a uniform bundle of splitting type 
             $(\underbrace{a_2,\dots,a_2}_{l_2},\dots, \underbrace{a_n,\dots,a_n}_{l_n})$
     \end{itemize}
     Then $W$ is isomorphic to $\mathcal{O}_X^{\oplus l_1}(a_1)$ where $\mathcal{O}_X(1)$ is the ample generator of Pic$(X)$. By induction, the bundle $U$ splits. Since $\text{Ext}^1(U,W)$ vanishes, the exact sequence $0\rightarrow W\rightarrow E\rightarrow U\rightarrow 0$ splits. So $E(\simeq U\oplus W)$ is a direct sum of line bundles. 
     
 \end{proof}

\appendix
\section{Additional details for Proposition 2.1}\label{appA}
In this appendix, we use notations as in the proof of Proposition \ref{mor from Q to Gr Constant} and provide more details by sketching the proof of \cite[Page 204, Case (ii)]{tango1976morphisms} and \cite[Lemma 3.3 (ii) and Case 1 of (iii)]{tango1974n}.\\

Suppose that there exists a non-constant morphism $f:\mathbb{Q}^{2n+1}\rightarrow \mathbb{G}(l,2n+1)$. Let $g(t)$ and $h(t)$ be the polynomials in Proposition \ref{mor from Q to Gr Constant}. We set
\begin{align*}
    g(t)=1-c_1t+c_2t^2+\cdots+(-1)^{l+1}c_{l+1}t^{l+1},\\
    h(t)=1+d_1t+d_2t^2+\cdots+d_{2n+1-l}t^{2n+1-l}.
\end{align*}
Let $a$ be $\sqrt[2n+2]{c_{l+1}d_{2n+1-l}}$. Then, we have the equation (see Page 3, the equation (\ref{eq for chern}))
\begin{align}\label{formula 1}
    g(t)h(t)=1+(-1)^{l+1}c_{l+1}d_{2n+1-l}t^{2n+2}=1+(-1)^{l+1}a^{2n+2}t^{2n+2}.
\end{align}

First, we use the same method of \cite[Lemma 3.3]{tango1974n} to prove that $a$ is a rational number.
\begin{proposition}\label{symmetry of polynomials}
    We have the identities $t^{l+1}g(\frac{1}{at})=(-1)^{l+1}g(\frac{t}{a})$ and $t^{2n+1-l}h(\frac{1}{at})=h(\frac{t}{a})$.
\end{proposition}
\begin{proof}
    Set $g(t)=(1-\alpha_1a t)(1-\alpha_2at)\cdots (1-\alpha_{l+1}at)$. By formula (\ref{formula 1}), we have
    \begin{align}\label{relations of roots}
        |\alpha_i|=1,~\alpha_i\ne \alpha_j \text{ if }i\ne j,~\alpha_i^{-1}=\overline{\alpha_i}\in \{\alpha_1,\alpha_2,\dots,\alpha_{l+1}\}.
    \end{align}
    Furthermore, $\alpha_1\cdots \alpha_{l+1}a^{l+1}$ equals $c_{l+1}$. Since both $a$ and $c_{l+1}$ are positive real numbers, we must have $\alpha_1\cdots \alpha_{l+1}=1$. Combining with (\ref{relations of roots}), we get
    \begin{align*}
        t^{l+1}g(\frac{1}{at})&=(t-\alpha_1)(t-\alpha_2)\cdots (t-\alpha_{l+1})=(\alpha_1^{-1}t-1)(\alpha_2^{-1}t-1)\cdots (\alpha_{l+1}^{-1}t-1)\alpha_1\cdots\alpha_{l+1}\\
        &=(\alpha_1t-1)(\alpha_2t-1)\cdots (\alpha_{l+1}t-1)=(-1)^{l+1}g(\frac{t}{a}).
    \end{align*}
    Similarly, we set $h(t)=(1+\alpha_1'a t)(1+\alpha_2'at)\cdots (1+\alpha_{2n+1-l}'at)$ and apply similar arguments to obtain $t^{2n+1-l}h(\frac{1}{at})=h(\frac{t}{a})$.
\end{proof}
\begin{corollary}\label{a is rational}
    The number $a$ is $\frac{c_{m+1}}{c_m}$, where $m$ is $\frac{l}{2}$.
\end{corollary}
\begin{proof}
    The equation $t^{l+1}g(\frac{1}{at})=(-1)^{l+1}g(\frac{t}{a})$ shows that $c_ia^{-i}$ equals $c_{l+1-i}a^{-l-1+i}$ for $1\leq i\leq l$. We get the desired conclusion by taking $i$ to be $\frac{l}{2}$.
\end{proof}

Once one proves that $a$ is rational, as in the proof of Proposition 2.1, both $a,C_i$ and $D_j$ are positive integers, we claim that\\

\textbf{Claim:} If $a,C_i$ and $D_j$ are all positive integers, we can apply the same proof of \cite[Case (ii)]{tango1976morphisms} to get a contradiction.\\

Now we sketch the proof of \cite[Page 204, Case (ii)]{tango1976morphisms}. Let $G(t)$ be $g(\frac{t}{a})(=1-C_1t+C_2t^2+\cdots+(-1)^{l+1}C_{l+1}t^{l+1})$ and $H(t)$ be $h(\frac{t}{a})(=1+D_1t+D_2t^2+\cdots+D_{2n+1-l}t^{2n+1-l})$. From formula (\ref{formula 1}), we have $H(t)G(t)=1+(-1)^{l+1}t^{2n+2}$.

\begin{lemma}
    Suppose that there exists a non-constant morphism $f$ from $\mathbb{Q}^{2n+1}$ to $\mathbb{G}(l,2n+1)$. Then we have the equations $n=l$,  $C_1=C_2=\cdots=C_{l}=D_1=D_2=\cdots=D_{2n-l}=2$ and $C_{l+1}=D_{2n+1-l}=1$.
\end{lemma}
\begin{proof}
    Since $\mathbb{G}(l,2n+1)$ is isomorphic to $\mathbb{G}(2n+1-l,2n+1)$, we may assume the inequality $n\geq l$. From the equation $H(t)G(t)=1+(-1)^{l+1}t^{2n+2}$ and the fact that $C_i,D_j$ are positive integers, the first part of Tango's proof allows us to conclude that $H(1)$ equals $2n+2$ and $D_1$ is bigger than $1$. 
    
    Next we prove that $D_1,D_2,\dots, D_{2n-l}$ are not less than $2$. Suppose that there exists a positive integer $k(\leq 2n-l)$ such that $D_k$ equals $1$. Let $r$ be $\min\{k|D_k=1\}$. Then both $D_1$ and $D_{r-1}$ are bigger than $1$. From $H(t)=t^{2n+1-l}H(\frac{1}{t})$ (Proposition \ref{symmetry of polynomials}), we obtain $D_{2n+1-l}=1$ and $D_{2n+1-i}=D_{i}$ for $1\leq i \leq 2n-l$. In particular, $D_{2n+1-r}$ is $D_r$. By the definition of $r$, we have $2r\leq 2n+1-l$. Let $Q$ be the universal quotient bundle of $\mathbb{G}(l,2n+1)$. By Pieri's formula, the class 
    \begin{align*}
         f^*\omega_{r,r,0,\dots,0}&=f^*(\omega_{r,0,\dots,0}^2-\omega_{r+1,0,\dots,0}\omega_{r-1,0,\dots,0})\\
         &=f^*(c_r(Q))^2-f^*(c_{r+1}(Q))f^*(c_{r-1}(Q))=(D_{r}^2-D_{r+1}D_{r-1})a^{2r}H^{2r}
    \end{align*}
    is the pullback of a Schubert cycle, which is numerically non-negative. (For the definition of the Schubert cycle $\omega_{a_0,\dots,a_{l}}$, we refer to \cite{tango1976morphisms}). So $D_r^2-D_{r+1}D_{r-1}\geq 0$. From the inequalities $D_{r}^2-D_{r+1}D_{r-1}\leq 1-2<0$, we get a contradiction. Hence, we have $2n+2=H(1)=1+D_1+\cdots+D_{2n+1-l}\geq 1+2(2n-l)+1=2(2n+1-l)$. Combining with the assumption $n\geq l$, we must have $n=l$, $D_1=D_2=\cdots=D_{2n-l}=2$ and $D_{2n+1-l}=1$. Finally, from $H(t)G(t)=1+(-1)^{l+1}t^{2n+2}$, we also have $C_1=C_2=\cdots=C_l=2$ and $C_{l+1}=1$.
\end{proof}

   Now we return to the proof of \textbf{Claim}. Since $l$ is even and $(l,2n+1)$ is not $ (2,5)$, we have $l\geq 4$ and $D_1=D_2=D_3=2$. Similar as above, we have $f^*\omega_{2,2,0,\dots,0}=f^*(\omega_{2,0,\dots,0}^2-\omega_{3,0,\dots,0}\omega_{1,0,\dots,0})=(D_2^2-D_3D_1)a^4H^4=0$. By Pieri's formula again, it shows that
   $$0=f^*(\omega_{2,2,0,\dots,0}\omega_{2n-1-l,0,\dots,0})=f^*\omega_{2n+1-l,2,0,\dots,0}+f^*\omega_{2n-l,2,1,0,\dots,0}+f^*\omega_{2n-l-1,2,2,0,\dots,0}.$$
   Since all the classes in the right hand side is numerically non-negative, we have $f^*\omega_{2n+1-l,2,0,\dots,0}=0$. On the other hand, we have
   \begin{align*}
       &f^*\omega_{2n+1-l,2,0,\dots,0}=f^*\omega_{2n-1-l,0,\dots,0}f^*\omega_{2,0,\dots,0}\\
       =&f^*(c_{2n+1-l}(Q))f^*(c_{2}(Q))=D_{2n+1-l}D_2a^{2n+3-l}H^{2n+3-l}\ne 0,
   \end{align*}
   which is a contradiction. %For more details, we refer to \cite[Page 204, Case (ii)]{tango1976morphisms}.

\section{Classification of primitive approximate solutions}
We use methods in \cite{EHS} to classify primitive approximate solutions. Let $P(T,X_1)=\sum_{k=0}^{2n}p_kX_1^kT^{2n-k}$ be a primitive approximate solution. 
\begin{proposition}\label{A1}
    If $p_{2n}$ is $0$, then $P(T,X_1)$ is $bT^{2n}$ for some rational number $b$.
\end{proposition}
\begin{proof}
    Let $S_{0}(T,X_1,X_2)$ be a symmetric divisor of $P(T,X_1)-R(X_1,X_2)$ such that for a $2(n+1)$-th primitive unit root $y_0$, we have $S_0(0,1,y_0)=0$. Since $y_0$ is a simple root of $P(0,1)-R(1,z)$, $y_0$ is also a simple root of $S_0(0,1,z)$. Therefore, by the implicit function theorem, there is a germ of holomorphic function $y(x)$ in a neighborhood of $x=0$ satisfying
    $$S_0(x(1+y(x)),1,y(x))=0 \text{ and } y(0)=y_0.$$
    We now show by induction that for $m=1,\dots,2n-1$, we have $y^{(m)}(0)=p_{2n-m}=0$. As  $S_0$ is symmetric in $X_1$ and $X_2$,  we have
    \begin{align}
        &P(x(1+y(x)),1)-R(y(x),1)=0, \tag{1}\label{eq1 for p2n zero} \\
        &P(x(1+y(x)),y(x))-R(y(x),1)=0. \tag{2}\label{eq2 for p2n zero}
    \end{align}
    By taking the derivatives of (\ref{eq1 for p2n zero}) and (\ref{eq2 for p2n zero}) at $x=0$ and noting that $p_{2n}$ is $0$, we obtain
    \begin{align}
        &p_{2n-1}(1+y_0)-R'(y_0,1)y'(0)=0, \tag{1'}\label{1d of eq1 for p2n zero} \\
        &p_{2n-1}(1+y_0)y_0^{2n-1}-R'(y_0,1)y'(0)=0. \tag{2'}\label{1d of eq2 for p2n zero}
    \end{align}
    From $R(y,1)(y^2-1)=y^{2n+2}-1$, we have $R'(y,1)(y^2-1)+R(y,1)\cdot(2y)=(2n+2)y^{2n+1}$. As $R(y_0,1)$ is $0$ and $y_0$ is not $ \pm 1$, $R'(y_0,1)$ is $\frac{(2n+2)y_0^{2n+1}}{y_0^2-1}(\ne 0)$. From (\ref{1d of eq1 for p2n zero}) and (\ref{1d of eq2 for p2n zero}), we get $p_{2n-1}(1+y_0)(y_0^{2n-1}-1)=0$. Since $y_0$ is primitive, $y_0^{2n-1}-1$ and $1+y_0$ are not zero. We have $p_{2n-1}=0$. By (\ref{1d of eq1 for p2n zero}) and noting that $R'(y_0,1)\ne0$, $y'(0)$ is  $0$.\\

    For the case $m\geq 2$, we assume by induction that $y^{(m')}(0)=p_{2n-m'}=0$ for $m'<m$. Then the $m$-th derivatives of (\ref{eq1 for p2n zero}) and (\ref{eq2 for p2n zero}) satisfy the following equations:
    \begin{align}
        &m!p_{2n-m}(1+y_0)^m-R'(y_0,1)y^{(m)}(0)=0, \tag{$1^m$}\label{md of eq1 for p2n zero} \\
        &m!p_{2n-m}(1+y_0)^my_0^{2n-m}-R'(y_0,1)y^{(m)}(0)=0. \tag{$2^m$}\label{md of eq2 for p2n zero}
    \end{align}
    Since $y_0$ is primitive, $y_0^{2n-m}-1$ is not zero. We have $y^{(m)}(0)=p_{2n-m}=0$ as above.
\end{proof}
\begin{proposition}\label{A2}
    If $p_{2n}$ is $1$, then $P(T,X_1)$ is $\Sigma_{n}(bT^2,X_1^2)$ for some $b\in \mathbb{Q}$.
\end{proposition}
\begin{proof}
    Let $S_{+}(T,X_1,X_2)$ be a symmetric divisor of $P(T,X_1)-R(X_1,X_2)$ such that for a $2n$-th primitive unit root $y_0$, we have $S_+(0,1,y_0)=0$. Note the equation $\Sigma_{n}(p_{2n-2}T^2,X_1^2)-R(X_1,X_2)=(p_{2n-2}T^2-X_2^2)\Sigma_{n-1}(p_{2n-2}T^2,X_1^2,X_2^2)$ (see \cite[Section 7.2]{EHS} for example). Let $S_-(T,X_1,X_2)$ be $\Sigma_{n-1}(p_{2n-2}T^2,X_1^2,X_2^2)$, then $S_-(0,1,y_0)$ is also $0$. Denote $P(T,X_1)$ by $P_+(T,X_1)$ and denote $\Sigma_{n}(p_{2n-2}T^2,X_1^2)$ by $P_-(T,X_1)$. We are going to show that $P_{+}$ equals $P_-$.
    Let $y_{\pm}(x)$ be germs of holomorphic functions satisfying
     $$S_\pm(x(1+y_\pm(x)),1,y_\pm(x))=0 \text{ and } y_\pm(0)=y_0.$$
     As $S_\pm$ is symmetric in $X_1$ and $X_2$, we have equations:
    \begin{align}
        &P_\pm(x(1+y_\pm(x)),1)-R(y_\pm(x),1)=0, \tag{$1_\pm$}\label{eq1 for p2n 1} \\
        &P_\pm(x(1+y_\pm(x)),y_\pm(x))-R(y_\pm(x),1)=0. \tag{$2_\pm$}\label{eq2 for p2n 1}
    \end{align}
    Let $p^+_{2n-m}$ (resp. $p^-_{2n-m}$) be the coefficient of $X_1^{2n-m}T^m$ in $P_+(T,X_1)$ (resp. $P_-(T,X_1)$). We prove  $p^+_{2n-m}=p^-_{2n-m}$ and $y^{(m)}_+(0)=y^{(m)}_-(0)$ by induction for $m$ $(0\leq m \leq 2n)$.

    When $m$ is $0$, we have $p_{2n}^+=p_{2n}^{-}=1$, $y_{+}(0)=y_-(0)=y_0$. We next show $p_{2n-1}^{+}=p_{2n-1}^-=0$ and $y'_{+}(0)=y'_{-}(0)=0$. Taking the derivatives of (\ref{eq1 for p2n 1}) and (\ref{eq2 for p2n 1}) at $x=0$, we get
    \begin{align}
        &p_{2n-1}^{\pm}(1+y_0)-R'(y_0,1)y'_{\pm}(0)=0, \tag{$1'_\pm$}\label{1d of eq1 for p2n 1} \\
        &p_{2n-1}^{\pm}(1+y_0)y_0^{2n-1}+2ny_0^{2n-1}y'_{\pm}(0)-R'(y_0,1)y'_{\pm}(0)=0. \tag{$2'_\pm$}\label{1d of eq2 for p2n 1}
    \end{align}
    From $R(y,1)(y^2-1)=y^{2n+2}-1$, we have $R'(y,1)(y^2-1)+R(y,1)\cdot(2y)=(2n+2)y^{2n+1}$. Note $R(y_0,1)=1$ (as $y_0$ is a $2n$-th primitive unit root). Substituting it in the above equation, we have
    \begin{align}
        2ny_0-y_0^2R'(y_0,1)=-R'(y_0,1). \tag{\dag}\label{rel for cal}
    \end{align}
    Multiplying (\ref{1d of eq2 for p2n 1}) by $y_0^2$ and using the relation (\ref{rel for cal}), one has
    \begin{align*}
        &p_{2n-1}^{\pm}(1+y_0)y_0-R'(y_0,1)y_{\pm}'(0)=0. \tag{$y_0^2\cdot 2'_\pm$}\\
        &p_{2n-1}^{\pm}(1+y_0)-R'(y_0,1)y_{\pm}'(0)=0, \tag{$1'_\pm$}
    \end{align*}
   Since $y_0$ is primitive, the number $ 
\left|\begin{array}{ccc} 
    y_0(1+y_0) &    -R'(y_0,1)    \\ 
    1+y_0 &    -R'(y_0,1)   
\end{array}\right| 
=(1-y_0)(1+y_0)R'(y_0,1)$ is not $0$. Then one obtains $p_{2n-1}^{\pm}=y_{\pm}'(0)=0$. \\

When $m$ is at least $2$, by induction, we may assume  $p^+_{2n-m'}=p^-_{2n-m'}$ and $y_{+}^{(m')}(0)=y_{-}^{(m')}(0)$ for $m'<m$. By taking the $m$-th derivatives of (\ref{eq1 for p2n 1}) and (\ref{eq2 for p2n 1}) at $x=0$, we have 
\begin{align}
        &m!p_{2n-m}^{\pm}(1+y_0)^m-R'(y_0,1)y^{(m)}_{\pm}(0)+T^{1}_{\pm}=0, \tag{$1^m_\pm$}\label{md of eq1 for p2n 1} \\
        &m!p_{2n-m}^{\pm}(1+y_0)^my_0^{2n-m}+2ny_0^{2n-1}y^{(m)}_{\pm}(0)-R'(y_0,1)y^{(m)}_{\pm}(0)+T^{2}_{\pm}=0, \tag{$2^m_\pm$}\label{md of eq2 for p2n 1}
    \end{align}
where $T^{i}_{\pm}$ are the remaining terms satisfying $T^1_-=T^1_+
$ and $T^2_-=T^2_+$. For the case $m=2$, by construction, we automatically have $p_{2n-2}^{-}=p_{2n-2}=p_{2n-2}^{+}$. Then one obtains $y_{+}^{(2)}(0)=y_-^{(2)}(0)$ from (\ref{md of eq1 for p2n 1}). Suppose now $m$ is at least $3$, multiplying (\ref{md of eq2 for p2n 1}) by $y_0^2$ and using (\ref{rel for cal}), we have 
\begin{align}
    m!p_{2n-m}^{\pm}(1+y_0)^my_0^{2n-m+2}-R'(y_0,1)y_{\pm}^{(m)}(0)+y_0^2T^{2}_{\pm}=0. \tag{$y_0^2\cdot 2^m_\pm$}\label{solve}
\end{align}
Note that $ 
\left|\begin{array}{ccc} 
    m!(1+y_0)^my_0^{2n-m+2} &    -R'(y_0,1)    \\ 
    m!(1+y_0)^m &    -R'(y_0,1)   
\end{array}\right| 
=m!(1-y_0^{2n-m+2})(1+y_0)^mR'(y_0,1)$ is not zero for $m\geq 3$. By solving the system of linear equations \{(\ref{solve}), (\ref{md of eq1 for p2n 1})\} (view $p_{2n-m}^{\pm}$ and $y_{\pm}^{(m)}(0)$ as indeterminate), we have $p_{2n-m}^+=p_{2n-m}^-$ and $y_-^{(m)}(0)=y_+^{(m)}(0)$.
\end{proof}
\begin{proposition}\label{A3}
    $\Sigma_n(T^2,X_1^2,X_2^2)$ is irreducible in $\mathbb{C}[T,X_1,X_2]$.
\end{proposition}
\begin{proof}
    It suffices to show that the variety $V:=\{(T,X_1,X_2)|\Sigma_n(T^2,X_1^2,X_2^2)=0\}\subset \mathbb{C}^3$ is smooth on $\mathbb{C}^3\backslash \{0\}$. By \cite[Lemma 7.2]{EHS}, the variety defined by $\Sigma_n(T,X_1,X_2)=0$ is smooth on $\mathbb{C}^3\backslash 0$. Note that the map $(T,X_1,X_2)\mapsto (T^2,X_1^2,X_2^2)$ is a local isomorphism outside the locus defined by $TX_1X_2=0$, $V$ is smooth on $\mathbb{C}^3\backslash \{TX_1X_2=0\}$.\\

    Now suppose $(t,u,v)(\ne (0,0,0))$ is a singular point of $V$, then one of $t,u$ and $v$ is $0$. By symmetry, we assume that $t$ is $0$. Then there are  equations
    \[\frac{\partial}{\partial T}\Sigma_n(T^2,X_1^2,X_2^2)|_{(0,u,v)}=\frac{\partial}{\partial X_1}\Sigma_n(T^2,X_1^2,X_2^2)|_{(0,u,v)}=\frac{\partial}{\partial X_2}\Sigma_n(T^2,X_1^2,X_2^2)|_{(0,u,v)}=0.\]
    Furthermore, we have  equations:
    \begin{align*}
        &\frac{\partial}{\partial X_1}\Sigma_n(T^2,X_1^2,X_2^2)|_{(0,u,v)}=\frac{\partial}{\partial X_1}\Sigma_n(X_1^2,X_2^2)|_{(u,v)}=0,\\
        &\frac{\partial}{\partial X_2}\Sigma_n(T^2,X_1^2,X_2^2)|_{(0,u,v)}=\frac{\partial}{\partial X_2}\Sigma_n(X_1^2,X_2^2)|_{(u,v)}=0.
    \end{align*}
    Without loss of generality, we may assume $v$ is not $0$. As $\Sigma_n(u^2,v^2)(=\Sigma_n(0,u^2,v^2))$  vanishes, $u/v$ is a root of $\Sigma_n(z^2,1)=0$ with multiplicity greater than $2$. 
    But $\Sigma_n(z^2,1)=\frac{z^{2n+2}-1}{z^2-1}$ has no multiple roots, which is a contradiction.
\end{proof}
\bibliography{ref}

\begin{thebibliography}{10}

\bibitem{du2021vector}
R.~Du, X.~Fang, and Y.~Gao.
\newblock Vector bundles on rational homogeneous spaces.
\newblock {\em Ann. Mat. Pura Appl. (4)}, 200(6):2797--2827, 2021.

\bibitem{EHS}
G.~Elencwajg, A.~Hirschowitz, and M.~Schneider.
\newblock Les fibres uniformes de rang au plus {$n$} sur {${\bf P}\sb{n}({\bf C})$} sont ceux qu'on croit.
\newblock In {\em Vector bundles and differential equations ({P}roc. {C}onf., {N}ice, 1979)}, volume~7 of {\em Progr. Math.}, pages 37--63. Birkh\"{a}user, Boston, MA, 1980.

\bibitem{Ellia}
P.~Ellia.
\newblock Sur les fibr\'es uniformes de rang {$(n+1)$}\ sur {${\bf P}\sp{n}$}.
\newblock {\em M\'em. Soc. Math. France (N.S.)}, (7):60, 1982.

\bibitem{FLL24}
X.~Fang, D.~Li, and Y.~Li.
\newblock Uniform bundles on generalized {G}rassmannians.
\newblock {\em International Mathematics Research Notices}, 2025(5):1--22, 2025.

\bibitem{Gro}
A.~Grothendieck.
\newblock Sur la classification des fibr\'{e}s holomorphes sur la sph\`ere de {R}iemann.
\newblock {\em Amer. J. Math.}, 79:121--138, 1957.

\bibitem{guyot1985caracterisation}
M.~Guyot.
\newblock Caract\'{e}risation par l'uniformit\'{e} des fibr\'{e}s universels sur la {G}rassmanienne.
\newblock {\em Math. Ann.}, 270(1):47--62, 1985.

\bibitem{KS}
Y.~Kachi and E.~Sato.
\newblock Segre's reflexivity and an inductive characterization of hyperquadrics.
\newblock {\em Mem. Amer. Math. Soc.}, 160(763):x+116, 2002.

\bibitem{munoz2012uniform}
R.~Mu\~{n}oz, G.~Occhetta, and L.~E. Sol{\'a}~Conde.
\newblock Uniform vector bundles on fano manifolds and applications.
\newblock {\em Journal f{\"u}r die reine und angewandte Mathematik}, 2012(664):141--162, 2012.

\bibitem{ottaviani1989some}
G.~Ottaviani.
\newblock Some extensions of {H}orrocks criterion to vector bundles on {G}rassmannians and quadrics.
\newblock {\em Annali di Matematica pura ed applicata}, 155:317--341, 1989.

\bibitem{Pan2015Tri}
X.~Pan.
\newblock Triviality and split of vector bundles on rationally connected varieties.
\newblock {\em Math. Res. Lett.}, 22(2):529--547, 2015.

\bibitem{PRT24}
A.~Patel, E.~Riedl, and D.~Tseng.
\newblock Moduli of linear slices of high degree smooth hypersurfaces.
\newblock {\em Algebra Number Theory}, 18(12):2133--2156, 2024.

\bibitem{sato1976uniform}
E.-i. Sato.
\newblock Uniform vector bundles on a projective space.
\newblock {\em Journal of the Mathematical Society of Japan}, 28(1):123--132, 1976.

\bibitem{tango1974n}
H.~Tango.
\newblock On $(n-1) $-dimensional projectlve spaces contained in the {G}rassmann variety $ {G}r (n, 1) $.
\newblock {\em Journal of Mathematics of Kyoto University}, 14(3):415--460, 1974.

\bibitem{tango1976morphisms}
H.~Tango.
\newblock On morphisms from projective space $ \mathbb{P}^n $ to the {G}rassmann variety $ {G}r (n, d) $.
\newblock {\em Journal of Mathematics of Kyoto University}, 16(1):201--207, 1976.

\bibitem{Van}
A.~Van~de Ven.
\newblock On uniform vector bundles.
\newblock {\em Math. Ann.}, 195:245--248, 1972.

\end{thebibliography}
\bibliographystyle{abbrv}
\end{document}